\documentclass[12pt]{iopart}
\usepackage[utf8]{inputenc}
\usepackage{etex}
\usepackage{lmodern}
\usepackage[english]{babel}

\usepackage{iopams,dsfont,mathrsfs,amsthm,amstext,amsopn}
\usepackage [normalem]{ulem}
\usepackage[noadjust]{cite}

\newcommand{\newstuff}[1]{{#1}}
\newcommand{\oldstuff}[1]{}

\theoremstyle{definition}
\newtheorem{definition}{Definition}[section]

\theoremstyle{remark}
\newtheorem{remark}[definition]{Remark}

\theoremstyle{plain}
\newtheorem{theorem}[definition]{Theorem}
\newtheorem{corollary}[definition]{Corollary}
\newtheorem{lemma}[definition]{Lemma}
\newtheorem{proposition}[definition]{Proposition}

\DeclareMathDelimiter{\lvert}{\mathopen}{symbols}{"6A}{largesymbols}{"0C}
\DeclareMathDelimiter{\rvert}{\mathclose}{symbols}{"6A}{largesymbols}{"0C}
\DeclareMathDelimiter{\lVert}{\mathopen}{symbols}{"6B}{largesymbols}{"0D}
\DeclareMathDelimiter{\rVert}{\mathclose}{symbols}{"6B}{largesymbols}{"0D}

\newcommand{\Lspace}[2][2]{{{L}}^{#1}(#2)} 

\newcommand{\Rd}{{\mathds{R}^3}}

\newcommand{\Cd}{{\mathds{C}^3}}

\newcommand{\norm}[2]{\left\lVert {#2} \right\rVert_{#1}} 
\newcommand{\pairing}[3][]{\left\langle {#2},{#3}\right\rangle_{#1}} 
\newcommand{\absval}[1]{\left\lvert{#1}\right\rvert} 
\newcommand{\rbra}[1]{\left({#1}\right)} 
\newcommand{\cbra}[1]{\left\{{#1}\right\}} 

\newcommand{\intx}[3][x]{\int_{#2} #3 \, \text{d}#1}

\DeclareMathOperator{\supp}{supp}

\DeclareMathOperator*{\argmin}{arg\,min}

\DeclareMathOperator{\dom}{dom}

\DeclareMathOperator{\vol}{vol}

\newcommand{\Rep}[1]{\Re\rbra{#1}}
\newcommand{\Imp}[1]{\Im\rbra{#1}}

\newcommand{\cube}[1]{\mathfrak{C}({#1})}
\newcommand{\ball}[1]{{\mathfrak{B}({#1})}}
\newcommand{\sobolev}[1]{\left(1+\absval{\gamma}^2\right)^{#1}}
\newcommand{\fourier}[2][\gamma]{\widehat{#2}(#1)}
\newcommand{\fourierdiff}[2]{{\left(\widehat{#1}-\widehat{#2}\right)}(\gamma)}
\newcommand{\dmax}{\delta_{\mathrm{max}}}

\newcommand{\Hmind}{{H^m}}
\newcommand{\Hsind}{{H^s}}
\newcommand{\CGreen}{c_1}
\newcommand{\CExGOS}{c_2}
\newcommand{\Cgen}{c_2}
\newcommand{\CFKwGOS}{c_3}
\newcommand{\Clambda}{c_4}
\newcommand{\sol}{f}
\newcommand{\solaldel}{\sol^{\delta}_{\alpha}}
\newcommand{\data}{g}
\newcommand{\ui}{u^{\rm i}}
\newcommand{\us}{u^{\rm s}}
\newcommand{\Rset}{\mathds{R}}
\newcommand{\Zset}{\mathds{Z}}

\newcommand{\paren}[1]{\left(#1\right)}
\newcommand{\bracket}[1]{\left[#1\right]}

\newcommand{\diffq}[2]{\frac{\partial #1}{\partial #2}}
\newcommand{\Xspace}{\mathcal{X}}
\newcommand{\Yspace}{\mathcal{Y}}
\newcommand{\solset}{\mathcal{D}}
\newcommand{\oalpha}{\alpha}

\usepackage[pdfstartview=FitH,breaklinks=true,bookmarksopen=true]{hyperref}
\hypersetup{colorlinks=true,citecolor=blue,linkcolor=blue}
\usepackage{bookmark}

\begin{document}
\title[variational source condition in inverse medium scattering]{Verification of a variational source condition for acoustic 
inverse medium scattering problems}

\author{Thorsten Hohage and Frederic Weidling}

\address{Institute for Numerical and Applied Mathematics, University of Goettingen, Lotzestr. 16-18, 37083 Goettingen, Germany}

\ead{f.weidling@math.uni-goettingen.de}

\begin{abstract}
	This paper is concerned with the classical inverse scattering problem to recover the refractive index of a medium given near or far field measurements of scattered time-harmonic acoustic waves. It contains the first rigorous proof of (logarithmic) rates of convergence for Tikhonov regularization under Sobolev smoothness assumptions for the refractive index. This is achieved by combining two lines of research, conditional stability estimates via geometrical optics solutions and variational regularization theory.
\end{abstract}


\submitto{\IP, version: \today}

\section{Introduction}
	Regularization theory deals with the approximate solution of ill-posed operator equations 
	\begin{equation*}
		F(\sol) = \data
	\end{equation*}
	in Hilbert or Banach spaces. In this paper we confine ourselves to Hilbert spaces $\Xspace$ and $\Yspace$, and $F$ maps from $\dom(F)\subset\Xspace$ to $\Yspace$. Let $\sol^{\dagger}\in\dom(F)$ denote the exact solution and $\data^{\delta}\in\Yspace$ noisy data satisfying  $\|F(\sol^{\dagger})-\data^{\delta}\|\leq \delta$. One of the most prominent methods to obtain stable approximations to $\sol^{\dagger}$ from such noisy data is Tikhonov regularization 
	\begin{eqnarray}\label{intro:eq:Tikhonov}
		\solaldel\in \argmin_{\sol\in \dom(F)}\bracket{\frac{1}{\alpha}\lVert F(\sol)-\data^{\delta}\rVert_{\Yspace}^2+\frac{1}{2} \lVert\sol\rVert_{\Xspace}^2}.
	\end{eqnarray}
A main question of regularization theory concerns the convergence 
$\|\sol_{\alpha}^{\delta}-\sol^{\dagger}\|_{\Xspace}\to 0$ as $\delta \to 0$ for appropriate choices
of $\alpha = \alpha(\delta,\data^{\delta})$,  
and the rate of this convergence. To obtain such rates, additional 
assumptions on $\sol^{\dagger}$ are required (see \cite[Prop. 3.11]{Engl1996}). 
For a long time such 
conditions were formulated as spectral source conditions of the form 
$\sol^{\dagger} = \varphi(F'[\sol^{\dagger}]^*F'[\sol^{\dagger}])w$ (see \cite{KNS:08} and numerous references therein) for 
some $w\in\Xspace$ and an index function $\varphi$, that is $\varphi:[0,\infty)\to[0,\infty)$ is a continuous increasing function  satisfying $\varphi(0)=0$. Starting with 
\cite{Hofmann2007} source conditions have been formulated in the form of 
variational inequalities
\begin{eqnarray}\label{intro:eq:varSC}
\fl\forall f\in \dom(F) \qquad 
		\frac{\beta}{2} \norm{\Xspace}{\sol^\dagger-\sol}^2 
\leq \frac{1}{2} \norm{\Xspace}{\sol}^2 - \frac{1}{2} \norm{\Xspace}{\sol^\dagger}^2+\psi\rbra{\norm{\Yspace}{F(\sol)-F(\sol^\dagger)}^2}
\end{eqnarray}
with some parameter $\beta\in(0,1]$, 
and this type of source conditions has become more and more popular 
in regularization theory (see e.g.\ \cite{Anzengruber2011, Flemming2011a, Flemming2012, Frick2012, Hofmann2012, Hohage2012, Werner2012, Anzengruber2013, Burger2013, Flemming2013, Grasmair2013, Grasmair2013a, Anzengruber2014, Cheng2014, Dunker2014, Hohage2014, Kaltenbacher2014a, Kaltenbacher2014}) due to the following advantages over spectral source conditions:
\begin{itemize}
\item Proofs based on variational source conditions tend to be much simpler 
than proofs based on spectral source conditions. E.g.,\ for Tikhonov 
regularization and concave index function $\psi$ a simple argument 
by Grasmair \cite{Grasmair2010} (see also \cite[Thm.~3.3]{Werner2012}) yields the convergence rate
\begin{equation}\label{intro:eq:rate}
\frac{\beta}{2}\norm{\Xspace}{f^\delta_{\oalpha}-f^\dagger}^2\leq 4\psi(\delta^2), 
\end{equation}
for an optimal choice of the regularization parameter $\oalpha$ fulfilling $-1/(2\oalpha)\in \partial(-\psi)(4\delta^2)$, where $\partial(-\psi)$ denotes the subdifferential of $-\psi$.
\item For bounded linear operators in Hilbert spaces variational source conditions are not only 
sufficient, but even necessary for certain rates of convergence for Tikhonov 
regularization and other regularization methods in many cases (see \cite{FHM:11}). 
For spectral source conditions this is 
only true in a supremum over $\sol^{\dagger}$ in certain smoothness classes, 
but not for individual $\sol^{\dagger}$. 
The result in \cite{FHM:11} were derived for so-called approximate source conditions, which 
have been shown to be equivalent to \eref{intro:eq:varSC} in \cite{Flemming2012,Flemming:12b}.
\item Variational source conditions do not involve the Fr\'echet derivative 
$F'$ of the forward operator. This reduces the regularity assumptions on $F$, 
but more importantly, it avoids restrictive assumptions on the relation of 
$F$ and $F'$ such as the tangential cone condition, which cannot be verified 
for most interesting nonlinear applications, e.g.\ inverse scattering problems. 
\item Variational source conditions can be used not only in a Hilbert, but 
also in a Banach space setting and for more general noise models 
and data fidelity terms (see e.g.\ \cite{Werner2012,Dunker2014,Hohage2014}). 
\end{itemize}
However, so far variational source conditions could be verified only 
in rather few cases: One option is to derive them from spectral source conditions, 
but then the variational approach does not yield additional information. 
For linear operators $F$ and $l^q$ penalties with respect to certain bases in the 
range of $F^*$, characterizations of variational source condition have been derived 
in \cite{Anzengruber2013,Burger2013}.  Moreover, reformulations of \eref{intro:eq:varSC} 
with $\psi(x)=\sqrt{x}$ for a phase retrieval and an option pricing problem were 
derived in \cite{Hofmann2007}.  
The purpose of this paper is to show for a classical inverse scattering problem 
that a variational source conditions holds true under Sobolev-type smoothness assumptions. 

The forward problem we consider is as follows: Given a refractive index $n=1-\sol$ and one or several incident wave(s) 
$\ui$ solving the Helmholtz equation $\Delta \ui+\kappa^2\ui=0$, determine the total field(s) $u= \ui+\us$ such that 
	\begin{eqnarray}
		\Delta u+\kappa^2nu = 0\qquad \mbox{in }\Rset^3,\label{intro:eq:diff}\\
		\diffq{\us}{r}-i\kappa \us=\mathcal{O}\paren{\frac{1}{r^2}}\qquad \mbox{as }r=\absval x \to \infty. \label{intro:eq:sommer}
	\end{eqnarray}
	We will study inverse problems to recover the refractive index given measurements of scattered fields as formulated precisely in Section \ref{sec:main_results}. Related problems occur in many applications in nondestructive testing, geophysical exploration, and x-ray imaging. 



	The main tool of our analysis are geometrical optics solutions introduced in Section~\ref{sec:GOS}.  
The use of such functions is  well established for deriving uniqueness and conditional stability estimates, which for convenience we write in the form
	\begin{eqnarray}\label{eq:stability}
		\forall \sol_1,\sol_2\in K\qquad 
		\frac{\beta}{2}
		\|\sol_1-\sol_2\|_{\Xspace}^2\leq \psi\left(\|F(\sol_1)-F(\sol_2)\|_{\Yspace}^2\right)
	\end{eqnarray}
	for certain (typically compact) smoothness classes $K\subset \dom(F)$, e.g.\ Sobolev balls. 
	For the acoustic inverse medium scattering problem the first such estimate was established by Stefanov \cite{Stefanov1990} using a very strong norm in the image space $\Yspace$ and a logarithmic function of the form 
	\[
		\psi(t) = C(\ln t^{-1})^{-2\mu}\qquad \mbox{for some }\mu>0.
	\]  
	This estimate was improved in \cite{Haehner2001} by choosing $\Yspace$ as an $L^2$ space and making the exponent $\mu$ explicit with $\mu\leq 1$. 
	Recently, improved estimates were established  in \cite{Isaev2013a} where $\mu\to\infty$ as the smoothness exponent of the Sobolev ball $K$ tends to $\infty$. Lower bounds in \cite{Mandache2001,Isaev2013c} show that such upper bounds are essentially optimal. The dependence on the wave number $\kappa$ was made explicit in 
\cite{Isakov2014,Isaev2014} leading to so-called H\"older-logarithmic stability estimates with significantly improved stability for large $\kappa$.
	
	To make this first paper on the verification of variational source conditions accessible to a larger audience and to keep the level of technicality as low as possible, we did not try to incorporate all of these recent improvements for stability estimates into our analysis. 

	Let us compare variational source conditions \eref{intro:eq:varSC} and conditional stability estimates \eref{eq:stability} on the abstract level of operator equations. Obviously, \eref{intro:eq:varSC} for all $\sol^{\dagger}\in K$ implies \eref{eq:stability}  since we may assume w.l.o.g.\ that $\|\sol_1\|\geq \|\sol_2\|$ and choose $\sol=\sol_1$ and $\sol^{\dagger}=\sol_2$. However, the reverse implication is not obvious since the case $\|\sol\|<\|\sol^{\dagger}\|$ cannot be excluded and since \eref{intro:eq:varSC} is required for all $\sol$ in the larger set $\dom(F)$, not only $\sol\in K$.


It is interesting to have not only a stability estimate \eref{eq:stability}, but also a variational source condition \eref{intro:eq:varSC} since the latter yields error bounds for reconstructions from noisy data obtained by Tikhonov regularization and other commonly used regularization methods. From conditional stability estimates one can also derive error bounds for the method of \emph{quasisolutions} $\sol^{\delta}_K \in \argmin_{\sol\in K} \lVert F(\sol) - \data^{\delta} \rVert_{\Yspace}$ \newstuff{ (see \cite{Ivanov:62}).} \oldstuff{, but}\newstuff{However,} this method is often difficult to implement and rarely used in practice 
\newstuff{ (see e.g.\ \cite{ACK:82} for a theoretical discussion for an inverse obstacle scattering 
problem)}.	
	Moreover,  the set $K$ must be known explicitly, a typically unrealistic assumption, whereas Tikhonov regularization with a-posteriori selection rules for the regularization parameter $\alpha$ can attain optimal rates of convergence over a large range of smoothness classes \emph{without} a-priori knowledge of the smoothness of the solution. 
	
\oldstuff{In summary, it seems a worthwhile endeavor to study which of the large variety of conditional stability 
estimates for various forward operators and smoothness classes can be sharpened to variational source conditions. 
The present paper gives a first example, but we expect that other results may follow. }

	The plan for the remainder of this paper is as follows: The next section contains a precise formulation of our main results. Section \ref{sec:GOS} describes properties of geometrical optics solutions, and Sections \ref{sec:near} and \ref{sec:farfield}  contain the proofs of our main results. 

\section{Main results}\label{sec:main_results}
We assume that the contrast $f= 1-n$ is supported in the ball $\ball \pi$ where $\ball R:=\{ x \in \Rd \colon \lvert x \rvert \leq R\}$ for $R>0$. 
Due to the physical constraints $\Re(n)\geq 0$ and $\Im(n)\geq 0$, the contrast $f$ belongs to the set  
\[
\solset:=\left\{f \in L^{\infty}(\Rset^3)\colon \Im(f) \leq 0, \Re(f)\leq 1, \supp(f)\subset {\ball{\pi}}\right\}.
\]  
	It is well known that for all $f\in\solset$ a solution $u$ to \eref{intro:eq:diff} exists, and it is unique due to \eref{intro:eq:sommer}. 
	Corresponding inverse problem consist in recovering $f$ from measurements of certain total fields $u$ for different incident fields 
	$u^\mathrm{i}$. We will discuss two such problems below. 
	
	The first one is to reconstruct $f$ from \emph{near field data}
	\begin{eqnarray*}
		w_f(x,y)=u^\mathrm{i}_y(x)+u^\mathrm{s}_y(x),\quad (x,y)\in \partial \ball R \times \partial \ball R,
	\end{eqnarray*}
	for some $R>\pi$. That is we measure for each incident point source wave of the form
	\begin{eqnarray*}
		u^\mathrm{i}_y(x)=\frac{1}{4\pi} \frac{\e^{i\kappa \absval{x-y}}}{\absval{x-y}}
	\end{eqnarray*}
	centered at $y\in \partial \ball R$  the corresponding total field on the sphere $\partial \ball R$. Our aim is to solve the equation 
	$F_\mathrm{n} (\sol)=w$ for given  data $w$, where 
	\begin{eqnarray*}
		F_\mathrm{n}\colon\solset \rightarrow  \Lspace{\partial \ball{R}^2},\quad f \mapsto w_f
	\end{eqnarray*}
	is the near field operator that maps the contrast $f$ to the Green's function $w_f$ of the pde. 
As preimage space $\Xspace$ we choose the Sobolev space $H^m_0(\ball \pi)$ equipped with the norm
\begin{eqnarray*}
	\norm{H^m}{f}:=\left(\sum_{\gamma\in \Zset^3} \sobolev{m} \absval{\fourier f}^2\right)^{1/2}
\end{eqnarray*}
where $\fourier f=(2 \pi)^{-3/2} \int_{\cube \pi} f(x) \e^{-i\gamma \cdot x} \mathrm d x$ are the Fourier coefficients of $f$ in the cube $\cube \pi= (-\pi,\pi)^3$ with $f$ extended by $0$ outside of $\ball \pi$. 
Our main result for near field data reads as follows:  
	
	\begin{theorem}[Variational source condition for near field data]\label{result:thm:sourcecon}
		Let $\frac{3}{2}< m<s$, $s\neq 2m+3/2$ and $\pi<R$. Suppose that the true contrast $f^\dagger$ satisfies $f^\dagger\in\solset$ with 
		$\lVert f^\dagger\rVert_\Hsind\leq C_s$ for some $C_s\geq 0$. Then a variational source condition \eref{intro:eq:varSC} 
		holds true for the operator $F_\mathrm{n}$ with 
		$\dom(F_\mathrm{n}):= \solset\cap\Xspace$, 
$\Yspace = \Lspace{\partial \ball{R}^2}$, 	$\beta=1/2$, and $\psi$ given by 
		\[
		\psi_\mathrm{n}(t):=A\left(\ln(3+t^{-1})\right)^{-2\mu}, \qquad \mu:=\min\left\{1,\frac{s-m}{m+3/2}\right\},
		\]
		where the constant $A>0$ depends only on $m,s,C_s,\kappa$, and $R$.
	\end{theorem}
	
	Actually, this theorem holds true for arbitrary $\beta\in (0,1)$, but the constant $A$ depends on the choice of $\beta$ 
	(see Remark \ref{rem:beta}).
	\newstuff{For a discussion of the case $s=2m+3/2$ see Remark \ref{rem:s=2m+3/2}.}
	From the result cited in \eref{intro:eq:rate} we obtain a rigorous error bound for Tikhonov regularization 
	without further assumptions on $F_\mathrm{n}$:
	
	\begin{corollary}\label{result:cor:ratenear}
		Under the assumptions of Theorem \ref{result:thm:sourcecon} the error bound 
		\begin{eqnarray*}
			\norm{\Hmind}{f^\delta_{\oalpha}-f^\dagger} \leq 4\sqrt{A} \rbra{ \ln (3+\delta^{-2})}^{-\mu}
		\end{eqnarray*}
    in terms of the noise level $\delta$ 
		holds true for the regularization scheme \eref{intro:eq:Tikhonov} if $\oalpha$ is chosen such that 
	$\frac{1}{2\oalpha}= A\frac{\partial \ln(3+t^{-1})^{-2\mu}}{\partial t}\big|_{t=4\delta^2}$.
	\end{corollary}
	
	We could formulate further corollaries from regularization theory, in particular on a-posteriori choice of the 
	regularization parameter. 
	Moreover, we obtain the following stability estimate:

	\begin{corollary}\label{result:cor:stabnear}
		Let $\frac{3}{2}<m<s$, $s\neq 2m+3/2$ and $\pi<R$. Let $f_1$ and $f_2$  satisfy $f_j \in \solset$ with $\lVert f_j \rVert_\Hsind\leq C_s$ for $j=1,2$ and some $C_s>0$. 
		Then the stability estimate
		\begin{equation}\label{eq:stabilityN}
			\norm{\Hmind}{f_1-f_2}\leq 2\sqrt{A} \rbra{\ln \left(3+\|F_\mathrm{n}(f_1)-F_\mathrm{n}(f_2)\|_{L^2(\partial \ball{R}^2)}^{-2}\right)}^{-\mu} 
		\end{equation}
		holds true with $\mu$ and $A$ as in Theorem \ref{result:thm:sourcecon}.
	\end{corollary}
	 \eref{eq:stabilityN} differs from similar results in the literature discussed 
	in the introduction in the use of a Sobolev norm rather than an $L^2$ or an $L^{\infty}$ norm on the left hand side. 
	However, this could easily be accommodated for in proofs of previous stability results, so the relevance of this work is 
	rather expressed in Corollary \ref{result:cor:ratenear} than in Corollary \ref{result:cor:stabnear}. 

\medskip
For the formulation of the second inverse problem recall that 
every solution $u$ of the Helmholtz equation \eref{intro:eq:diff} fulfilling the Sommerfeld radiation condition \eref{intro:eq:sommer}
has the asymptotic behavior
	\begin{eqnarray*}
		u(x)=u^\mathrm{i}(x)+\frac{\e^{i \kappa r}}{r}\rbra{u^\infty(\hat x)+ \mathcal O\rbra{\frac{1}{r^2}}},\quad r=\absval x\rightarrow \infty
	\end{eqnarray*}
	uniformly for all directions $\hat x = x/r\in \mathfrak S^2:=\{x\in \Rd\colon \absval x=1\}$, and $u^{\infty}$ is called the far field pattern 
	(see \cite{Colton2013,Kirsch2011}). Often far field patterns $u^\infty (\cdot,d)$ corresponding to incident plane waves 
	$u^\mathrm{i}_{d,\infty}(x)= \e^{i \kappa x\cdot d}$ propagating in direction $d$ are considered as data. 
	This leads to an inverse problem for the forward operator 
	\begin{eqnarray*}
		F_\mathrm{f}\colon\solset\cap \Xspace \rightarrow \Lspace{\mathfrak S^2\times\mathfrak S^2},\quad f\mapsto u^\infty.
	\end{eqnarray*}

	\begin{theorem}[Variational source condition for far field data]\label{result:thm:sourcefar}
		Let the assumptions of Theorem \ref{result:thm:sourcecon} be satisfied. Then the operator 
		$F_\mathrm{f}$ with $\dom(F_\mathrm{f}):= \solset\cap\Xspace$ and $Y=\Lspace{\mathfrak S^2\times\mathfrak S^2}$ fulfills for all 
		$0<\theta<1$ a variational source condition \eref{intro:eq:varSC} with $\beta=1/2$ and $\psi$ given by 
		\[
		\psi_\mathrm{f}(t):=B\left(\ln(3+t^{-1})\right)^{-2\mu\theta}
		\]
		with $\mu$ given as in Theorem \ref{result:thm:sourcecon}, and a constant $B>0$  depending only on 
		$m,s,C_s,\kappa,\theta$, and $R$.
	\end{theorem}
	
	Once again we obtain a convergence rate via \eref{intro:eq:rate} and a stability estimate as corollary: 
	
	\begin{corollary}\label{result:cor:far}
Suppose the assumptions of Theorem \ref{result:thm:sourcefar} hold true. 
	\begin{enumerate}
		\item Then the minimizers of the Tikhonov functional \eref{intro:eq:Tikhonov} satisfy the error bound
		\begin{eqnarray*}
			\norm{\Hmind}{f^\delta_{\oalpha}-f^\dagger} \leq 4\sqrt{B} \rbra{ \ln (3+\delta^{-2})}^{-\mu\theta}
		\end{eqnarray*}
		if	$\frac{1}{2\oalpha}= B\frac{\partial \ln(3+t^{-1})^{-2\mu\theta}}{\partial t}\big|_{t=4\delta^2}.$
		\item For all $f_1, f_2\in\solset$ satisfying $\lVert f_j \rVert_\Hsind\leq C_s$ for $j=1,2$, $s\neq2m+3/2$ and some $C_s>0$ 
		the following stability estimate holds true:
		\begin{eqnarray*}
			\norm{\Hmind}{f_1-f_2}\leq& 2\sqrt{B} \rbra{\ln\left(3+\lVert F_\mathrm{f}(f_1)-F_\mathrm{f}(f_2)\rVert_{\Lspace{\mathfrak S^2\times\mathfrak S^2}}^{-2}\right)}^{-\mu\theta}. 
		\end{eqnarray*}
	\end{enumerate}
	\end{corollary}

\section{Bounding Fourier coefficients using geometrical optics solutions}\label{sec:GOS}

As mentioned in the introduction, the key tool of our proof are \emph{geometrical optics solutions}. These functions are solutions 
of the perturbed Helmholtz equation \eref{intro:eq:diff} of the form
	\begin{equation}\label{eq:geom_opt}
		u(x)=\e^{i\zeta x}(1+v(x,\zeta)),
	\end{equation}
	where $\zeta \in \Cd\setminus\Rd$ satisfies $\zeta \cdot \zeta=\kappa^2$ (where $\zeta\cdot\xi:=\sum_{j=1}^3\zeta_j\xi_j$ for 
$\zeta,\xi\in\Cd$).  
Such exponentially growing solutions were introduced by Faddeev in \cite{Faddeev1965} and have been used to prove uniqueness of 
electrical impedance tomography and inverse scattering problems \cite{Calderon1980,Sylvester1987,Novikov1988} 
as well as for stability estimates \cite{Alessandrini1988,Stefanov1990,Haehner2001,Isakov2014,Isaev2013a,Isaev2014} for inverse scattering 
for space dimension $d\geq 3$. We recommend the textbook \cite{Kirsch2011} or the monograph \cite{Colton2013} 
for concise and self-contained introductions  
of the version of geometrical optics solutions used below and refer to the monograph \cite{Isakov2006} and the review \cite{Uhlmann2009} 
for numerous extensions and further references. 
	
A reason for the interest in geometrical optics solutions is the following lemma:
\begin{lemma}[{\cite[Lemma 3.2]{Haehner2001}}] \label{source:lem:lowfreq}
		Let $\pi<R<R^\prime$. Then there exists a positive constant $c_1$ (depending on $\kappa, R$, and $R^\prime$) such that 
		for all contrasts $f_1, f_2\in\solset$ with corresponding near fields $w_1$ and $w_2$, and 
		for all solutions $u_j\in H^{2}(\ball{R^\prime})$ to $\Delta u_j+\kappa^2 u_j=\kappa^2 f_j u_j$ 
		the following estimate holds true:
		\begin{eqnarray*}
			\fl \absval{ \intx{\ball{\pi}}{\rbra{f_1-f_2}u_1 u_2}}  \leq \CGreen  \norm{\Lspace{\partial \ball{R}^2}}{w_{1}-w_{2}} \norm{\Lspace{\ball{R^\prime}}}{u_1} \norm{\Lspace{\ball{R^\prime}}}{u_2}. 
		\end{eqnarray*}
	\end{lemma}
	Note that if we could choose $u_1$ and $u_2$ as geometrical optics solutions with $\zeta_1=\overline{\zeta_2}$ and $v_1=v_2=0$, this lemma would immediately yield bounds on the Fourier coefficients of $f_1-f_2$. 
Even though such a choice is impossible for $f_1,f_2\not\equiv 0$,  we will derive bounds on the small Fourier coefficients 
of $f_1-f_2$ in Lemma \ref{source:lem:lowmain} using bounds on $v$ detailed below.

	\begin{proposition}[Existence and norm estimate of geometrical optics solutions]\label{geoopt:thm:geooptexist}
		Let $R^\prime >\pi$ and $f\in \solset$. Then for all $\zeta \in \Cd$ with $\zeta \cdot \zeta= \kappa^2$ such that $t:=\absval{\Imp \zeta}$ fulfills $ t \geq 2 \kappa^2 \frac{R^\prime}{\pi} \norm{\Lspace[\infty]{\Rd}}{f}$ there exists a function $v(x,\zeta)$ such that $u(x,\zeta):=\e^{i\zeta \cdot x}\rbra{1+v(x,\zeta)}$ belongs to 
		$H^{2}(\ball{R^\prime})$, solves the equation $\Delta u + \kappa^2 u= \kappa^2 f u$ in $\ball{R^\prime}$, and satisfies the estimates
		\begin{eqnarray}
			\norm{\Lspace{\ball{R^\prime}}}{v(\cdot, \zeta)} \leq  \frac{\CExGOS}{ t} \norm{\Lspace[\infty]{\Rd}}{f} \label{geoopt:lem:geooptprop:eq:vl2kappa}\\
			\norm{\Lspace{\ball{R^\prime}}}{v(\cdot, \zeta)}\leq \Cgen \label{geoopt:lem:geooptprop:eq:vl2un}\\
			\norm{\Lspace{\ball{R^\prime}}}{u(\cdot, \zeta)}\leq \Cgen \e^{R^\prime \absval{\Imp \zeta}}\label{geoopt:lem:geooptprop:eq:ul2}
		\end{eqnarray} 
		with a positive constant $\CExGOS$ depending on $\kappa, R$, and $R^\prime$.
	\end{proposition}
	\begin{proof}
		For the mapping properties see \cite[Lemma 5]{Haehner1996}, the first estimate can be found in the proof of  \cite[Lemma 2.9]{Haehner1998}. To prove \eref{geoopt:lem:geooptprop:eq:vl2un}, insert the lower bound $2 \kappa^2\frac{R^\prime}{\pi} \norm{\Lspace[\infty]{\Rd}}{f}$ for $t$ into \eref{geoopt:lem:geooptprop:eq:vl2kappa}. As 
		\begin{eqnarray*}
			\norm{\Lspace{\ball{R^\prime}}}{u(\cdot, \zeta)}&=&\norm{\Lspace{\ball{R^\prime}}}{e^{i\zeta \cdot x}\rbra{1+v(x,\zeta)}}\\
			&\leq& \norm{\Lspace[\infty]{\ball{R^\prime}}}{e^{-(\Im \zeta) \cdot x}}\rbra{\norm{\Lspace{\ball{R^\prime}}}{1}+\norm{\Lspace{\ball{R^\prime}}}{v(x,\zeta)}},
		\end{eqnarray*}
		one obtains \eref{geoopt:lem:geooptprop:eq:ul2} by inserting \eref{geoopt:lem:geooptprop:eq:vl2un}.
	\end{proof}

	In the following the constant $M_{\mathrm{em}}$ is given by the Sobolev embedding theorem such that 
	\begin{eqnarray}
		\norm{\Lspace[\infty]{\overline{\ball{\pi}}}}{f}
\leq M_{\mathrm{em}} \norm{\Hmind}{f}\label{source:lem:lowfreq:eq:Memdef}
	\end{eqnarray}
	for all $f \in H_0^m(\ball \pi)$. The following estimate is similar to the weaker result \cite[Lemma 3.3]{Haehner2001} 
	where the factor $\|f_1-f_2\|_{H^m}$ on the right hand side is missing. A similar result can be found in 
	\cite[Lemma 2.1]{Isaev2014}, but its proof relies on more sophisticated results on geometrical optics solutions.  
	
	\begin{lemma}\label{source:lem:lowmain}
		Let $C_m>0$, $m>3/2$, $\pi<R<R^\prime$ and $f_1$ and $f_2$ be contrasts with $f_j \in \solset$ and $\lVert f_j \rVert_\Hmind \leq C_m$  
		with corresponding near field data $w_j$ for $j=1,2$. Define
		\begin{eqnarray}\label{source:lem:lowmain:eq:deft0}
			t_0:= 2 \kappa^2\frac{{R^\prime}}{\pi} M_{\mathrm{em}} C_m
		\end{eqnarray}
		with $M_{\mathrm{em}}$ defined as in \eref{source:lem:lowfreq:eq:Memdef}. 
		Let $t\geq t_0$ and $1\leq\varrho\leq 2\sqrt{\kappa^2+t^2}$. Then there exists a constant $\CFKwGOS>0$ depending only on $m, R, R^\prime,$
		\oldstuff{$\kappa$ and $\varrho$} \newstuff{and $\kappa$} 
		such that for all $\gamma \in \Zset^3$ satisfying $\absval{\gamma}\leq \varrho$  we have
		\begin{eqnarray*}
			\absval{\fourierdiff{f_1}{f_2}} \leq& \CFKwGOS \e^{4 R^\prime t}  \norm{\Lspace{\partial \ball{R}^2}}{w_{1}-w_{2}}  +\frac{\CFKwGOS}{t}\norm{\Hmind}{f_1-f_2}.
		\end{eqnarray*}
	\end{lemma}
	\begin{proof}
		For fixed $\gamma \in \Zset^3$ choose two unit vectors $d_1$ and $d_2$ in $\Rd$ 
such that $\gamma \cdot d_1=\gamma \cdot d_2= d_1 \cdot d_2=0$. For $t>t_0$ define
		\begin{eqnarray*}
			\zeta_t &:=-\frac{1}{2} \gamma +i t d_1+\sqrt{\kappa^2+t^2-\frac{\absval{\gamma}^2}{4}} d_2,\\
			\eta_t &:=-\frac{1}{2} \gamma -i t  d_1-\sqrt{\kappa^2+t^2-\frac{\absval{\gamma}^2}{4}} d_2.
		\end{eqnarray*}
		Then $\zeta_t, \eta_t \in \Cd$ satisfy
		\begin{eqnarray*}
			\zeta_t+\eta_t=-\gamma, \qquad \absval{\Im(\zeta_t)}=\absval{\Im(\eta_t)}\geq t_0, \qquad \zeta_t \cdot \zeta_t = \eta_t \cdot \eta_t =\kappa^2.
		\end{eqnarray*}
		Hence by Theorem \ref{geoopt:thm:geooptexist} there exist geometrical optical solutions of the form
		\begin{eqnarray*}
			u_1 (x,\zeta_t)&= \e^{i \zeta_t \cdot x} \rbra{1+ v_1(x,\zeta_t)},\\
			u_2 (x,\eta_t)&= \e^{i \eta_t \cdot x} \rbra{1+ v_2(x,\eta_t)},
		\end{eqnarray*}
		where $u_j$ solves the equation $\Delta u_j+\kappa^2  u_j=\kappa^2 f_j u_j$ in $\ball{R^\prime}$ for $j=1,2$.
 It follows that
		\begin{eqnarray}
			\fl\absval{\fourierdiff{f_1}{f_2}}&=&\frac{1}{(2\pi)^{3/2}} \bigg\lvert\intx{\ball{\pi}}{(f_1-f_2)(x)\, \e^{-i\gamma \cdot x}}\bigg\rvert \nonumber \\
			&=&\frac{1}{(2\pi)^{3/2}} \bigg\lvert \intx{\ball{\pi}}{(f_1-f_2)(x)\, u_1 (x,\zeta_t)\, u_2 (x,\eta_t)} -\int_{\ball{\pi}} \Big[(f_1-f_2)(x)  \nonumber \\
			& &  \rbra{v_1(x,\zeta_t)+ v_2(x,\eta_t)+v_1(x,\zeta_t)\, v_2(x,\eta_t)} \e^{-i\gamma \cdot x} \Big]\, \text{d}x \bigg\rvert.\label{source:lem:lowmain:eq1}
		\end{eqnarray}
		
		The first integral on the right hand side of \eref{source:lem:lowmain:eq1} can be bounded by Lemma \ref{source:lem:lowfreq} and by the norm estimate \eref{geoopt:lem:geooptprop:eq:ul2} for geometrical optics solutions for $t\geq t_0$:
		\begin{eqnarray}
			\absval{\intx{\ball{\pi}}{(f_1-f_2)(x)\, u_1 (x,\zeta_t)\, u_2 (x,\eta_t)}} \leq  \CFKwGOS \norm{\Lspace{\partial \ball{R}^2}}{w_{1}-w_{2}}\e^{2 R^\prime t}.\label{source:lem:lowmain:eq2}
		\end{eqnarray}
		
Using \eref{geoopt:lem:geooptprop:eq:vl2kappa} and \eref{geoopt:lem:geooptprop:eq:vl2un} 
the second integral on the right hand side of \eref{source:lem:lowmain:eq1} can be estimated by 
		\begin{eqnarray}
			& &\absval{\intx{\ball{\pi}} {\e^{-i\gamma \cdot x} (f_1-f_2)(x)\, \rbra{v_1(x,\zeta_t)+ v_2(x,\eta_t)+v_1(x,\zeta_t)\, v_2(x,\eta_t)}}}\nonumber\\
			&\leq& \norm{\Lspace{\ball{\pi}}}{f_1-f_2} \rbra{ \norm{\Lspace{\ball{R^\prime}}}{v_1(\cdot,\zeta_t)}+\norm{\Lspace{\ball{R^\prime}}}{v_2(\cdot,\eta_t)}}+ \nonumber\\
			& &+\norm{\Lspace[\infty]{\ball{\pi}}}{f_1-f_2} \rbra{\norm{\Lspace{\ball{R^\prime}}}{v_1(\cdot,\zeta_t)}\norm{\Lspace{\ball{R^\prime}}}{v_1(\cdot,\zeta_t)}}  \nonumber\\
			&\leq& \frac{\CFKwGOS}{t}  \norm{\Hmind}{f_1-f_2} \label{source:lem:lowmain:eq3}.
		\end{eqnarray}		
		Plugging \eref{source:lem:lowmain:eq2} and \eref{source:lem:lowmain:eq3} into \eref{source:lem:lowmain:eq1} yields the assertion.
	\end{proof}
	To avoid another free parameter, we will set $R^\prime:=2R$ in the rest of this paper.

\section{Proof of Theorem \ref{result:thm:sourcecon}}\label{sec:near}
	We will prove the following equivalent formulation of the variational source condition \eref{intro:eq:varSC}:
	\begin{eqnarray}\label{result:thm:sourcecon:eq}
		\eqalign{
		& \Re\pairing[\Hmind]{f^\dagger}{f^\dagger-f}=\Rep{\sum\limits_{\gamma\in \Zset^3} \sobolev{m} \fourier{f^{\dagger}} \overline{\fourierdiff{f^{\dagger}}{f}}}\\
		\leq& \frac{1-\beta}{2}\norm{\Hmind}{f^\dagger - f}^2+\psi\rbra{\norm{\Yspace}{F_\mathrm{n}(\sol)-F_\mathrm{n}(\sol^\dagger)}^2}
		}
	\end{eqnarray}
	This form has also been used in regularization theory (see e.g.\ \cite{Hofmann2007}), but we preferred to formulate 
	our main theorems using the form \eref{intro:eq:varSC} since its relation to stability results is more obvious 
	and since the form \eref{intro:eq:varSC} is used in the proof of the convergence rate \eref{intro:eq:rate}.

	
We bound the high Fourier coefficients on the left hand side of \eref{result:thm:sourcecon:eq} as follows: 
	\begin{lemma}\label{source:lem:highfreq}
		If $\varrho>0$,  $0 \leq m\leq s$, $f^{\dagger} \in H_0^s(\ball \pi)$ and $\sol \in H_0^m(\ball \pi)$, then 
		\begin{eqnarray*}
	\fl	\Rep{\sum\limits_{\gamma\in \Zset^3\setminus \ball \varrho} \sobolev{m} \fourier{f^{\dagger}} \overline{\fourierdiff{f^{\dagger}}{\sol}}} \leq& \frac{1}{8} \norm{\Hmind}{f^{\dagger}-\sol}^2 
	+ 2\norm{\Hsind}{f^{\dagger}}^2 \varrho^{2(m-s)}.
		\end{eqnarray*}
	\end{lemma}
	\begin{proof}
		By the Cauchy-Schwarz inequality and Young's inequality $xy\leq 2x^2+\frac{1}{8}y^2$ for $x,y\geq 0$ we have
		\begin{eqnarray}
			& &\Rep{\sum\limits_{\Zset^3\setminus \ball \varrho} \sobolev{m} \fourier{f^{\dagger}} \overline{\fourierdiff{f^{\dagger}}{\sol}}} \nonumber \\
			&\leq&   \sqrt{\sum\limits_{\Zset^3\setminus \ball \varrho} \sobolev{m} \absval{\fourierdiff{f^{\dagger}}{\sol}}^2}
			\sqrt{\sum\limits_{\Zset^3\setminus \ball \varrho} \sobolev{m} \absval{\fourier{f^{\dagger}}}^2}\label{source:lem:highfreq:eq1}\\
&\leq& \frac{1}{8}\sum\limits_{\Zset^3\setminus \ball \varrho} \sobolev{m} \absval{\fourierdiff{f^{\dagger}}{\sol}}^2
  +  2\sum\limits_{\Zset^3\setminus \ball \varrho} \sobolev{m} \absval{\fourier{f^{\dagger}}}^2.\nonumber 
		\end{eqnarray}
		The first sum on the right hand side is bounded by $\frac{1}{8}\lVert{f^{\dagger}-\sol}\rVert_\Hmind^2$. 
		To bound the other sum we use that $f^{\dagger}$ is smoother than $\sol$ to obtain
		\begin{eqnarray*}
		\fl	\sum\limits_{\Zset^3\setminus \ball \varrho} \sobolev{m} \absval{\fourier{f^{\dagger}}}^2 &\leq \frac{1}{(1+\varrho^2)^{s-m}} \sum\limits_{\Zset^3\setminus \ball \varrho} \sobolev{s} \absval{\fourier{f^{\dagger}}}^2
			\leq  \varrho^{2(m-s)} \norm{\Hsind}{f^{\dagger}}^2.
		\end{eqnarray*}
		Inserting these estimates into \eref{source:lem:highfreq:eq1} completes the proof.
	\end{proof}
Obviously this bound is only useful with a proper choice of $\varrho$ as detailed below. To bound the 
low Fourier coefficients we use Lemma \ref{source:lem:lowmain}:
\begin{proof}[Proof of Theorem \ref{result:thm:sourcecon}]
Given $f\in H^m_0(\ball \pi)\cap \solset$  we distinguish two cases:\\
\emph{Case 1:} $\lVert{f^\dagger-f}\rVert_\Hmind > 4C_s$. By the Cauchy-Schwarz inequality we have
 		\begin{equation}\label{source:proof:eq:easypart}
		\eqalign{
 			\Re\pairing[\Hmind]{f^\dagger}{f^\dagger-f}&\leq\norm{\Hmind}{f^\dagger }\norm{\Hmind}{f^\dagger - f}\nonumber\\
 			&\leq C_s \norm{\Hmind}{f^\dagger - f}\leq \frac{1}{4}\norm{\Hmind}{f^\dagger - f}^2,
			}
 		\end{equation}
 		which clearly implies \eref{result:thm:sourcecon:eq}.

\emph{Case 2:} $\lVert{f^\dagger-f}\rVert_\Hmind \leq 4C_s$. 
Then $\norm{\Hmind}{f}\leq 5 C_s$, and hence we can apply Lemma \ref{source:lem:lowmain} with $C_m=5C_s$ and 
$t_0$ as defined in \eref{source:lem:lowmain:eq:deft0}. Moreover, we choose $R^\prime=2R$, $t \geq t_0$, and  $1\leq\varrho\leq 2\sqrt{\kappa^2+t^2}$ 
and set $\delta:= \lVert w_{f}-w_{f^\dagger} \rVert_{\Lspace{\partial \ball{R}^2}}$ to obtain
		%
\begin{equation}\label{source:proof:eq3}
\fl\eqalign{
			\Rep{\sum\limits_{\Zset^3\cap \ball \varrho}\sobolev{m} \fourier{f^\dagger} \overline{\fourierdiff{f^\dagger}{f}}}
			\leq & \CFKwGOS \sum\limits_{\gamma \in \Zset^3\cap \ball \varrho} \absval{\sobolev{m} \fourier{f^\dagger}} \times\\
			 &\times\rbra{ \e^{4Rt}\delta  +\frac{1}{t}\norm{\Hmind}{f-f^\dagger}}\cr.
			}
\end{equation}
		The sum on the right hand side of \eref{source:proof:eq3} can bounded by the Cauchy-Schwarz inequality and Lemma \ref{source:lem:rho} below with $\lambda=2m-s$ as follows 
		\begin{equation}\label{source:proof:eq4}
		\fl\eqalign{
			\sum\limits_{\gamma \in \Zset^3\cap \ball \varrho} \absval{\sobolev{m} \fourier{f^\dagger}} 
			&\leq  \sqrt{\sum\limits_{\gamma \in \Zset^3} \sobolev{s} \absval{\fourier{f^\dagger}}^2} \sqrt{\sum\limits_{\gamma \in \Zset^3\cap \ball \varrho} \sobolev{2m-s}} \\
	 &\leq \Clambda   C_s  \varrho^\tau
	}%
		\end{equation}
		with $\tau=\max\{2m+3/2-s,0\}$. Putting Lemma \ref{source:lem:highfreq}, \eref{source:proof:eq3}, and \eref{source:proof:eq4} together yields 
		\begin{eqnarray}
\fl & &			\Re{\pairing[\Hmind]{f^\dagger}{f^\dagger-f}} \nonumber\\
\fl&\leq & \frac{1}{8} \norm{\Hmind}{f^\dagger-f}^2 	+ 2C_s^2 \varrho^{2(m-s)}+\CFKwGOS\Clambda C_s\e^{4 R t}  \varrho^{\tau} \delta		
+ \frac{\CFKwGOS\Clambda C_s\varrho^{\tau}}{t}\norm{\Hmind}{f^\dagger-f} \label{source:proof:eq:5}\\
\fl			&\leq & \rbra{\frac{1}{8}+\frac{1}{8}\frac{\varrho^{2\tau+2(s-m)}}{\varepsilon t^2}} \norm{\Hmind}{f^\dagger-f}^2 
			+ \CFKwGOS\Clambda C_s \e^{4 R t}  \varrho^{\tau} \delta + 2C_s^2\rbra{1+ \varepsilon\CFKwGOS^2\Clambda^2}   \varrho^{2(m-s)} \nonumber
		\end{eqnarray}
		for all $\varepsilon>0$ by Young's inequality. Now we have to choose the free parameters $t,\varrho$ and $\varepsilon$ such that 
		the right hand side of the last inequality is approximately minimal for given $\delta>0$ and the constraints in Lemma \ref{source:lem:lowmain} 
		are satisfied. We pick $t,\varrho$ and $\varepsilon$ such that 
		\begin{equation}\label{eq:choices}
			12Rt = \ln(3+\delta^{-2})= \varrho^{\tau+s-m},\qquad \varepsilon:= (12R)^2.
		\end{equation}
		As $\tau+s-m\geq m+3/2>3$, there exists $\overline{t}>0$ such that $2t\geq (12Rt)^{1/(\tau+s-m)}$ for all $t\geq \overline{t}$.   
		Let us strengthen the constraint $t\geq t_0$ in Lemma~\ref{source:lem:lowmain} to $t\geq \tilde{t}_0$ with 
		$\tilde{t}_0:=\max\{t_0,\overline{t}\}$. Then
		the constraint $2\sqrt{\kappa^2+t^2}> 2t\geq (12Rt)^{1/(\tau+s-m)}=\varrho$ in Lemma~\ref{source:lem:lowmain} 
		holds true for all $t\geq \tilde{t}_0$, and $\rho\geq 1$ is satisfied as well. However, with the choice \eref{eq:choices} the constraint $t\geq \tilde{t}_0$ 
		itself is only satisfied for $\delta\leq \dmax$ with $\dmax:=  (\exp(12R\tilde{t}_0)-3)^{-1/2}$ (or $\dmax:=\infty$ 
		if $\exp(12R\tilde{t}_0)\leq 3$). 
		The case $\delta>\dmax$ will be treated at the end of this proof. 
		 Plugging \eref{eq:choices} 	into \eref{source:proof:eq:5} yields
		\begin{eqnarray*}
		\fl\Re{\pairing[\Hmind]{f^\dagger}{f^\dagger-f}} 
		&\leq& \frac{1}{4}\norm{\Hmind}{f^\dagger-f}^2 + 2C_s^2(1+\varepsilon\CFKwGOS^2\Clambda^2)\left(\ln(3+\delta^{-2})\right)^{-2\frac{s-m}{\tau+s-m}}\\
		&&+ \CFKwGOS\Clambda C_s(3+\delta^{-2})^{1/3}\delta \left(\ln(3+\delta^{-2})\right)^{\tau/(\tau+s-m)}\\
		&\leq& \frac{1}{4}\norm{\Hmind}{f^\dagger-f}^2+\tilde{A}\left(\ln(3+\delta^{-2})\right)^{-2\mu}
		\end{eqnarray*}
		for $\delta\leq \dmax$ with some constant $\tilde{A}$ depending only on  $m,s,C_s,\kappa$, and $R$ since the term in the second line tends to $0$ faster than 
		the last term in the first line as $\delta \searrow 0$. 
		This shows \eref{result:thm:sourcecon:eq} for $\delta\leq \dmax$.

		It remains to treat the case $\delta>\dmax$. By the Cauchy-Schwarz and Young's inequality we have
		\begin{eqnarray}\label{source:proof:eq:extbasic}
			\Re\pairing[\Hmind]{f^\dagger}{f^\dagger-f}
			&\leq&  \frac{1}{4}\norm{\Hmind}{f^\dagger-f}^2+\norm{\Hmind}{f^{\dagger}}^2\\
			&\leq& \frac{1}{4}\norm{\Hmind}{f^\dagger-f}^2+ C_s^2
		\end{eqnarray}
		Hence,  Theorem \ref{result:thm:sourcecon} holds true with $A:=\max\{\tilde{A},C_s^2(\ln(3+\dmax^{-2}))^{2\mu}\}$. 
	\end{proof}
	
		\begin{remark}\label{rem:beta}
		Note that in the previous proof both contributions to the term $\frac{1-\beta}{2}\lVert {f^\dagger- f}\rVert_{H^m}^2$ in \eref{result:thm:sourcecon:eq} 
		originate from an application of Young's inequality $xy\leq  \frac{1}{2\varepsilon}x^2+\frac{\varepsilon}{2}y^2$. Hence, by a proper 
		choice of $\varepsilon>0$ Theorem \ref{result:thm:sourcecon} can be shown for arbitrary $\beta\in (0,1)$. 
		\end{remark}

It remains to show the following lemma needed in the previous proof: 
	
	\begin{lemma}\label{source:lem:rho}
		For all $\lambda\neq -3/2$ there exists $\Clambda>0$ depending on $\lambda$ such that 
		\begin{eqnarray}
			\sqrt{\sum\limits_{\gamma \in \Zset^3\cap \ball \varrho} \sobolev{\lambda}} \leq \Clambda  \varrho^{\tau}\label{source:lem:rho:eq}
		\end{eqnarray}
		for all $\rho\geq 1$ where $\tau=\max\{\lambda+3/2,0\}$.
	\end{lemma}
	\begin{proof}
	For each $\gamma \in \Zset^3$ we introduce a cube  $\mathfrak C_\gamma:=\{x\in\Rset^3:|x-\gamma|_{\infty}\leq 1/2\}$ around $\gamma$ with 
	side length $1$. Then $\bigcup_{\gamma\in\Zset}^3 \mathfrak C_\gamma = \Rset^3$, and two different cubes intersect at most on a set of 
	measure $0$. Consider first the case that $\lambda\geq 0$. 
	As $\bigcup_{\gamma \in \Zset^3\cap \ball \varrho}\mathfrak C_\gamma \subset \ball {\rho+\sqrt{3}}$, it follows that 
	$\#\{\Zset^3\cap \ball \varrho\}\leq \vol\left(\ball {\rho+\sqrt{3}}\right)$, so 	we can estimate
		\begin{equation*}
			\sum\limits_{\gamma \in \Zset^3\cap \ball \varrho} \sobolev{\lambda}
			\leq \frac{4\pi}{3} (\varrho+\sqrt{3})^3 (1+\varrho^2)^{\lambda}\leq c\varrho^{3+2\lambda}\qquad  \mbox{for }\varrho\geq 1.
		\end{equation*}
		Now consider the case $\lambda < 0$. For $\gamma\neq 0$ note that $\mathfrak C_\gamma\cap \ball{|\gamma|}$ contains the ball 
		with radius $1/4$ centered at $\gamma-\frac{1}{4}\operatorname{sgn}(\gamma)$, the volume of which is $V=\pi/48$. Hence we obtain
		\begin{eqnarray*}
			\fl\sum\limits_{\gamma \in \Zset^3\cap \ball{\varrho}} \sobolev{\lambda}
			&\leq 1+\frac{1}{V}\sum\limits_{\gamma \in \Zset^3\setminus\cbra{0}} \int_{\mathfrak C_\gamma\cap \ball \varrho} \rbra{1+\absval{x}^2}^{\lambda}\, \mathrm d x\\
			&\leq 1+\frac{4\pi}{V} \int_{1/2}^{{\varrho}}r^2 \rbra{1+r^2}^{\lambda}\, \mathrm d r
			\leq 1+192 \int_{1/2}^{{\varrho}}r^{2+2\lambda}\, \mathrm d r\\	
	&\leq \left\{\begin{array}{l l}
		 1+\frac{192}{3+2\lambda}\varrho^{3+2\lambda},& 0>\lambda>-3/2,\\[0.5ex]
		 1+ \frac{192}{(-2\lambda-3)2^{3+2\lambda}},  & \lambda<-3/2.
\end{array}\right.
		\end{eqnarray*}
		This shows the lemma for all values of $\lambda\neq -3/2$.
	\end{proof}
	
\newstuff{\begin{remark}\label{rem:s=2m+3/2}
It is easy to see that for $\lambda = -3/2$ the left hand side of \eref{source:lem:rho:eq} grows 
logarithmically in $\lambda$. Using this fact, it is straightforward to treat the  
case $s=2m+3/2$ in Theorem \ref{result:thm:sourcecon} in a similar way, which leads 
to a function $\psi_\mathrm{n}$ involving an additional logarithmic term. 
Since we believe that this upper bound is not optimal anyways, we did not consider it worth to 
include the computations for this special case. Note that if $f^\dagger\in H^s$ for 
$s=2m+3/2$ then also $f^\dagger \in H^{s^\prime}$ for all $s>s^\prime$, and therefore 
Theorem \ref{result:thm:sourcecon} still holds true with $s$ replaced by $s^\prime$ for all 
$s>s^\prime>m$. 
		\end{remark}
		}

\section{Proof of Theorem \ref{result:thm:sourcefar}}\label{sec:farfield}
	In this section we prove the extension of the result for near field data to far field data. The main tool 
is the following lemma relating near and far field data:
\begin{lemma}[{\cite[Section 4]{Haehner2001}}]\label{source:lem:neartofar}
Let $R>\pi$, $m>3/2$, $C_m>0$ and $0<\theta<1$. Then there exist constants $\omega,\varrho,\dmax>0$ such 
that for any two contrasts $f_1,f_2\in \solset$ satisfying $\norm{\Hmind}{f_j} \leq C_m$ we have 
\begin{eqnarray}
			\norm{\Lspace{\partial \ball{\newstuff{2}R}^2}}{w_2-w_1}^2 \leq  \varrho^2 \exp \rbra{- \rbra{- \ln \frac{\norm{\Lspace{\mathfrak S^2 \times \mathfrak S^2}}{u^\infty_2-u^\infty_1}}{\omega \varrho}}^\theta }\label{source:lem:neartofar:eq}
\end{eqnarray}
if $\norm{\Lspace{\mathfrak S^2 \times \mathfrak S^2}}{u^\infty_2-u^\infty_1}\leq \dmax$ 
 where $w_j$ and $u^\infty_j$ denote near and far field scattering data for $f_j$,  $j=1,2$. 
\end{lemma}

\begin{proof}[Proof of Theorem \ref{result:thm:sourcefar}]
We distinguish the same cases as in the proof of Theorem \ref{result:thm:sourcecon}:\\
\emph{Case 1:} $\lVert{f^\dagger-f}\rVert_\Hmind > 4C_s$. This case can be treated as in the proof of Theorem \ref{result:thm:sourcecon}. \\
\emph{Case 2:} $\lVert{f^\dagger-f}\rVert_\Hmind \leq 4C_s$. Then  $\norm{\Hmind}{f}\leq 5 C_s$. 
Hence Lemma \ref{source:lem:neartofar} can be applied with $C_m=5 C_s$. 
Setting $\delta:= \lVert F_\mathrm{f}(f^\dagger)- F_\mathrm{f}(f)\rVert_{\Lspace{\mathfrak S^2 \times \mathfrak S^2}}$ 
and $\varphi(t):= \varrho^2\exp(-(-\ln(\sqrt{t})+\ln(\omega\varrho))^{\theta})$, it follows from Theorem 
\ref{result:thm:sourcecon} 
and the monotonicity of $\psi_\mathrm{n}$ that the variational source condition \eref{intro:eq:varSC} holds true with 
$\psi(t) = \psi_\mathrm{n}(\varphi(t))$ if $\delta\leq \dmax$. 
Bounding $\psi_\mathrm{n}(t)\leq A(\ln t^{-1})^{-2\mu}$ for $t<1$ we obtain 
\[
\fl\psi_\mathrm{n}(\varphi(t)) \leq  
A\left(-\left(-\ln(\sqrt{t})+\ln (\varrho\omega)\right)^{\theta}-\ln\varrho^2\right)^{-2\mu}
\quad \mbox{for } \sqrt{t}\leq\min\left\{\dmax,\frac{1}{2}\right\}.
\]
Hence, it is easy to see  that there are constants $B>0$ and $\widetilde{\delta}_\mathrm{max}\in (0,
\min\{\dmax,\frac{1}{2}\}]$ 
such that 
\[
\psi_\mathrm{n}(\varphi(t))\leq  B(\ln(3+t^{-1}))^{-2\mu\theta}\quad \mbox{for }\sqrt{t}\leq \widetilde{\delta}_\mathrm{max}.
\]
This shows  \eref{intro:eq:varSC} for $\delta\leq \widetilde{\delta}_\mathrm{max}$.
The case $\delta>\widetilde{\delta}_\mathrm{max}$ can be treated as in the proof of Theorem \ref{result:thm:sourcecon}, see eq.\ \eref{source:proof:eq:extbasic}. 
	\end{proof}

\newstuff{\section{Conclusions and Outlook}
We have established a source condition in the form of a variational inequality for three-dimensional 
acoustic medium scattering problems which implies logarithmic rates of convergence of Tikhonov 
regularization under Sobolev smoothness assumptions as the noise level tends to zero. 
Our proof is based on geometrical optics solutions, a technique that has been developed to show 
uniqueness and stability results for a number of parameter identification problems in differential 
equations.}

\newstuff{Our results may be extended in several directions: One may study other differential equations, 
in particular other inverse scattering problem such as electromagnetic or elastic scattering 
problems (\cite{Haehner1998}), electrical impedance tomography (\cite{Alessandrini1988}) or 
two-dimensional acoustic scattering problems (\cite{bukhgeim:08}). Moreover, analogs to improved 
stability estimates with explicit dependence on the wave-number and exponents tending to $\infty$ 
with the Sobolev smoothness index would be of interest (\cite{Isaev2014}). Finally, variational 
source conditions based on Banach norms would be desirable. So far we do not see a general scheme 
for establishing variational source conditions based on conditional stability estimates, 
it rather seems that each extension requires new ideas.}

\newstuff{Even beyond the technique of geometrical optics solutions it seems a worthwhile endeavor to 
study which of the large variety of conditional stability estimates for various forward operators 
and smoothness classes can be sharpened to variational source conditions. 
}
\phantomsection\addcontentsline{toc}{section}{References}
\section*{References}
\bibliographystyle{abbrv}
\bibliography{ver_lit}

\end{document}